\documentclass[twocolumn]{autart}    
\let\theoremstyle\relax

\usepackage[utf8]{inputenc}
\usepackage{comment}
\usepackage{algorithm}
\usepackage{algpseudocode}
\usepackage{breqn}
\usepackage{graphicx}
\usepackage[version=4]{mhchem}
\usepackage{siunitx}
\usepackage{longtable,tabularx}
\usepackage{graphics} % for pdf, bitmapped graphics files
\usepackage{float}
\usepackage{epsfig} 
\usepackage{times} 
\usepackage{amsmath} 
\usepackage{amssymb}  
\usepackage{mathtools}
\usepackage{pifont}
\usepackage[dvipsnames]{xcolor}
\usepackage{amsthm}
\usepackage{enumitem}
\usepackage{bm}
\usepackage{bbm}
\usepackage{soul}

\newcommand{\trace}{\textit{tr}}

\newcommand{\E}{\mathbb{E\,}}
\newcommand{\Cov}{\text{cov}} % bob unbolded and lower-cased this

\newcommand{\R}{\mathbb{R}}

\newtheorem*{definitiondes}{Definition (design cost)}
\newtheorem{lemma}{Lemma}
\newtheorem*{remark}{Remark}
\newtheorem{theorem}{Theorem}
\newtheorem{proposition}{Proposition}
\newtheorem{corollary}{Corollary}
\DeclareMathOperator*{\argmin}{arg\,min}

\newcommand{\python}{\mbox{\textsc{Python}}}
 
\definecolor{paleGreen}{rgb}{.3, .7, .3}
\definecolor{coolBlue}{rgb}{.3, .5, 1}
\definecolor{rosePink}{rgb}{.9, .5, .4}
\definecolor{ghost}{rgb}{.8, .8, .8}

\begin{document}

\begin{frontmatter}

\title{State estimation for control: an approach for output-feedback stochastic MPC} 

\author[UCSD,SDSU]{Mohammad S. Ramadan}\ead{msramada@eng.ucsd.edu},    
\author[UCSD]{Robert R. Bitmead}\ead{rbitmead@eng.ucsd.edu}  

and \author[SDSU]{Ke Huang}\ead{khuang@sdsu.edu }

\address[UCSD]{Department of Mechanical \&\ Aerospace Engineering, University of California, San Diego, La Jolla CA 92093-0411, USA.}   

\address[SDSU]{Electrical \&\ Computer Engineering Department, San Diego State University, San Diego CA 92182, USA.}  %


\begin{abstract}
The paper provides a new approach to the determination of a single state value for stochastic output feedback problems using paradigms from Model Predictive Control, particularly the distinction between open-loop and closed-loop control and between deterministic optimal control and stochastic optimal control. The State Selection Algorithm is presented and relies on given dynamics and constraints, a nominal deterministic state-feedback controller, and a sampling based method to select the best state value, based on optimizing a prescribed finite-horizon performance function, over the available candidates provided by a particle filter. The cost function is minimized over the horizon with controls determined by the nominal controller and the selected states. So, the minimization is performed not over the selection of the control other than through the choice of state value to use. The algorithm applies generally to nonlinear stochastic systems and relies on Monte Carlo sampling and averaging. However, in linear quadratic polyhedrally constrained cases the technique reduces to a quadratic program for the state value. The algorithm is evaluated in a set of computational examples, which illustrate its efficacy and limitations. Numerical aspects and the opportunity for parallelization are discussed. The examples demonstrate the algorithm operating, in closed-loop with its attendant particle filter, over the long horizon.
\end{abstract}
\end{frontmatter}

\section{Introduction}
We consider state selection for output feedback control of nonlinear stochastic systems. Starting with a filtered particle density of the state and a nominal control law, an algorithm is presented to choose a single state value from the density, which is then used in the control. The algorithm proceeds from the repeated selection of a single particle followed by a closed-loop stage, which generates a set of candidate control sequences, which then are evaluated for open-loop performance averaged over all the particles. This appearance of open-loop and closed-loop calculations highlights a significant distinction between deterministic and stochastic optimal feedback control, as is the propagation of single estimates versus whole densities. Both are discussed shortly. This also brings to the fore a difficulty in the formulation of output feedback stochastic MPC, which provides a touchstone control problem from which to appreciate the methods and, in part, their genealogy.

\subsection*{Context and MPC}
The surge of interest and research in Model Predictive Control (MPC) since the late 1980s is in large measure attributable to its capacity to handle constraints \cite{mayne2014model}. It does this within a receding-horizon optimal control context, which solves repeatedly a constrained open-loop optimal control problem from the current state, presumed to be available. That is, MPC is formulated as full-state feedback with the feedback being achieved by the receding horizon device of applying only the first element of the open-loop control sequence, each element of which is a function of the initial state, before measuring the next state and re-solving the open-loop problem. Being open-loop, each of these constrained optimization problems is manageable in its complexity. 

MPC using partially observed states enters into the realms of stochastic optimal control, which is known to be computationally intractable in all but the simplest of cases \cite{fel1966optimal,bar1981stochastic}, with or without constraints. Further, the presence of an equivalent open-loop optimal solution sequence for these problems evanesces; optimal controls are necessarily feedback only \cite{kumar2015stochastic}. Our study here is to walk a middle path between the in computational intractability (and optimality) of stochastic optimal control and the convenience of open-loop methods. Given a specific state feedback control law, $u_k=\kappa(x_k)$, as might be provided, say, by explicit MPC,  we develop an heuristic (and evidently suboptimal) approach to the selection of a single state value from the state particle filtered conditional density. The control law $\kappa$ thus effected relies on the filtered density solely through the selection process. For unconstrained Linear Quadratic problems, the selection of the conditional mean of the particle density would be optimal \cite{aastrom2012introduction}. However for constrained and/or nonlinear problems, state selection depends on the given control law and evaluation, in open-loop and on the particle ensemble, of both constraint satisfaction and a performance measure.  The evaluation of candidate state vectors builds on the finite-horizon open-loop nature of the MPC iteration and combines two aspects along a prediction horizon: the probability of constraint violation and the calculation of a predicted performance function, both averaged over the particle density, as is explicated later. 

Stochastic optimal control with partially observed state, which we abbreviate to \textit{stochastic optimal control}, is solved using Stochastic Dynamic Programming \cite{kumar2015stochastic,bertsekas2012dynamic} which requires propagation of full conditional densities at each time.
 The optimal control law necessarily is feedback-only; it cannot be precomputed along a horizon because it depends on anticipated measurements and the associated conditional density propagation. This should be compared to deterministic optimal control (and implicitly MPC methods) where the open-loop optimal control can be computed, say via Pontryagin's Maximum Principle, and coincides as a function of time with the feedback optimal control of dynamic programming. It is this feature of deterministic optimal control, which receding-horizon MPC relies upon for the generation of a feedback control law from an open-loop constrained optimization. The optimization, since it is open-loop, is computationally achievable. These concepts of open- and closed-loop underpin the computational issues in stochastic optimal control.

Because of the proscriptive burden of stochastic optimal control, several less cumbersome but approximate and suboptimal approaches to these problems have been developed.
\begin{itemize}
\item [\ding{93}] Wide-sense approaches, in which state uncertainty is presented by approximates of its first two central moments, and an approximate dynamics by linearization or Taylor expansion, of the propagation of these moments with the input and anticipated measurements, is used \cite{bar1981stochastic}.
\item [\ding{93}] Explicit probing, that is, including a state variance-based expression explicitly in the cost of a deterministic version of the stochastic problem, as in \cite{la2016dual}.
\item [\ding{93}] Stochastic MPC approaches in which active learning, or the reliance on future measurements, is typically dropped. This permits achieving an open-loop solution. Example stochastic MPC approaches that are designed to deal with initial state uncertainty include:
\begin{itemize}
    \item [\ding{75}] a modified version of scenario-based approaches directed towards initial state uncertainty \cite{sehr2017particle},
    \item [\ding{75}] sequential Monte Carlo approaches for sampling the input space, for nonlinear but unconstrained problems \cite{kantas2009sequential},
    \item [\ding{75}] stochastic tube approaches \cite{yan2005incorporating,cannon2012stochastic,heirung2018stochastic}, which by their formulation are limited in applicability to linear systems.
\end{itemize}  
\end{itemize}
A more complete review of stochastic MPC approaches can be found in \cite{mesbah2016stochastic}. Here, we present an algorithm akin to the stochastic MPC approaches (or equivalently to open-loop feedback control \cite{bertsekas2012dynamic}). Differently from these approaches, we seek to select the initial state value rather than designing the controller. At variance from stochastic tube approaches, our formulation extends naturally to constrained nonlinear dynamics. This is due to the sampling approach we use, which, in contrast to \cite{kantas2009sequential}, is parallelizable over the samples.

As discussed by Striebel in \cite{striebel1965sufficient}, purpose or function should guide data reduction efforts. Our approach is to focus on the selection of a single state vector value from the ensemble of particles representing the density. This reduction is influenced by a stated control objective, with a given -- as yet, neither necessarily stabilizing, nor optimal nor constraint-feasible -- nominal control law. So the control value, guided by the control objective, is indirectly determined by the nominal law and the selection process. This differs from the other methods which directly select the control value. The method proceeds from the following information: the system state equation with white process noise, a nominal state-feedback control law, the constraints, the noise density, a filtered particle state density, a cost function along the prediction horizon, state and input constraints, and their violation rate tolerance. Running alongside this control computation is the propagation of the filtered state particle density, whose properties also affect the closed-loop behavior.

The incorporation of the initial state as a decision variable in MPC was used previously in the literature, for linear dynamics with full state-feedback. In \cite{schluter2022stochastic}, the initial state is augmented to the decision variable of a tube-based MPC problem, to increase the likelihood of establishing feasibility. Also, in \cite{mayne2005robust}, the previous concept of augmentation is used, but to achieve constant value function over an invariant set around the origin. our approach does not assume a known initial state and extends naturally to nonlinear dynamics.

In the linear problem with linear nominal feedback control, quadratic cost and polyhedral constraints, the solution developed by the state selection approach coincides with a quadratic program on the initial states.

\subsection*{Outline}
The precise problem formulation and our proposed algorithm are given in Section~\ref{section: ProblemFormulation}. Section~\ref{section: properties} provides some rudimentary properties of our approach and relates it to stochastic MPC. The computational workload, including the propagation of the state particle state filter, and the possibility of parallel computing of the proposed algorithm are explained in Section~\ref{Section:NonlinearSystemsAlgorithm}. The algorithm's special case, for linear systems with polyhedral constraints, is discussed in Section~\ref{Section:LinearPolyhedron}, along with a result showing that, in the unconstrained case, the algorithm yields the conditional mean. These properties of the algorithm in the linear, quadratic, polyhedral case are the strongest technical support for the algorithm. Section~\ref{Section:examples} includes numerical examples for the nonlinear case and linear special case.

\section{Problem Formulation} \label{section: ProblemFormulation} 
At the current time $k=0$, we possess the following information.
\begin{enumerate}[label=\Roman*.]
\item Discrete-time state dynamics
\begin{align}\label{NonSys}
x_{k+1}=f(x_k,u_k,w_k),
\end{align}
where state $x_k\in\mathbb R^{r_x}$, control input $u_k\in\mathbb R^{r_u}$, exogenous disturbance $w_k\in\mathbb R^{r_w}$.
\item Stochastic disturbance process, $\{w_k\}$, assumed to be independent and identically distributed (i.i.d.) possessing known density $\mathcal W$. State $x_0$ is independent of $w_k$ for all $k$.
\item State and input constraint sets, $\mathbb X$ and $\mathbb U$, respectively, and $\epsilon\in [0,1)$, an acceptable probabilistic constraint violation rate.
\item Initial state $x_0$-density $p_0$, provided as a collection of particles $\Xi=\{\xi^i_0\in\mathbb R^{r_x}, j=1,\dots,L\}$.
\item Nominal full-state-feedback control law $u_k=\kappa(x_k),$  which is (so far) neither assumed to be recursively feasible with respect to $\mathbb X$ and $\mathbb U$ nor optimal with respect to the following, or indeed any, cost.\footnote{While the nominal control law need not be feasible nor optimal, its selection is material for the performance and feasibility of state selection. This is demonstrated in Section~\ref{Section:examples} by evaluation of the state selection algorithm with several controllers.}
\item An $N$-stage finite-horizon trajectory cost function $J$
\begin{align*}
    J = \sum_{k=0}^N \ell_k(x_k,u_k).
\end{align*}
\end{enumerate}

Because the initial state $x_0$ is not precisely known (unless $p_0$ is a point mass function), we seek to select a candidate state value, $x^\star_0$, supported by the initial density which yields: feasibility, perhaps probabilistically or statistically in simulation, on the finite horizon; and, a favorable influence on the subsequent evaluation of the trajectory cost $J$ in ensemble average over the particles.

If the dynamics in (\ref{NonSys}) have $w_ks$ zero and $x_0$ known, we refer to this as the \textit{deterministic} case.

Next we propose an algorithm, \textit{the State Selection Algorithm}, which returns a \textit{candidate state} $x_0^\star$, which is fed to the control $u_0=\kappa(x_0^\star)$ and applied to the system. The system output is used to update the particle filter density and the process repeated with new current time $k=0$. This will be shown by the examples provided in Section~\ref{Section:examples}.
\section*{State Selection Algorithm}
\begin{enumerate}
\item \label{step:alpha} Select sample repetition number $M$ and statistical feasibility tolerance $\alpha \in [0,\epsilon)$, parameters of the algorithm.
%    \item \cut{Given the set of $L$ particles $\Xi$, pick the  }
    
    \item For each $i \in \{1,2,\hdots,L\}$, choose $x_0'=\xi^i\in\Xi$.\label{step2} 
    \begin{enumerate}
        \item \label{step:lambda} With $(Nr_w)$-vectors
        \begin{align*}
        W'_j&=\begin{pmatrix}w^{\prime^T}_{0,j}&\hdots&w^{\prime^T}_{N-1,j}\end{pmatrix}^T,\\
        W''_j&=\begin{pmatrix}w^{\prime\prime^T}_{0,j}&\hdots&w^{\prime\prime^T}_{N-1,j}\end{pmatrix}^T,
        \end{align*}
sample, from $\mathcal{W}^N$ and $\Xi$, $M$ independent realizations of the $(2Nr_w+r_x)$-vector sequence 
        $$\left \{\begin{pmatrix}
            W'_{j}\\W''_{j}\\x_{0,j}'' \end{pmatrix}:\, j=1,\hdots,M\right\}.$$
            
            \item \label{step:explore}For each $j$ and $k$, compute the $\kappa$-closed-loop state sequence from $x_0'$, that is, $x_{0,j}'=x_0'$.
\begin{align}
    x_{k+1,j}'=f(x_{k,j}',\kappa(x_{k,j}'),w_{k,j}'). \label{eq:dashrec}
\end{align}
This defines the $M$ samples of the $N$-long closed-loop control sequence $\{\kappa(x_{k,j}')\}_{k=0}^{N-1}$ from the current initial state candidate $x_0'$.
\item For each $j\in\{1,\dots,M\}$ and $k\in\{0,N-1\}$, compute the open-loop-controlled state sequence from $x''_{0,j}$,
\begin{align}
    x_{k+1,j}''=f(x_{k,j}'',\kappa(x_{k,j}'),w_{k,j}'').\label{eq:ddashrec}
\end{align}
        \item\label{betaCondition} Compute for each $k$ the sample-average closed-loop control violation rate,\footnote{$\mathbbm{1}(\cdot)$ is the indicator function of an event.}
        \begin{align}\label{eq:betahat}
        \hat \beta_k(x_0')=\frac{1}{M}\sum_{j=1}^M\mathbbm{1}(\kappa(x_{k,j}') \in \mathbb{U}).
        \end{align}
        \item\label{lambdaCondition} Compute for each $k$ the sample-average open-loop-controlled state violation rate,
        \begin{align}\label{eq:lambdahat} 
        \hat \lambda_k(x_0')=\frac{1}{M}\sum_{j=1}^M\mathbbm{1}(x_{k,j}'' \in \mathbb{X}).
        \end{align}
        \item If $\hat \beta_k(x_0') \geq 1-\alpha$ and $\hat \lambda_k(x_0') \geq 1-\alpha$, for all $k$, then declare this candidate state, $x'_0$, to be \textit{feasible} and proceed. Otherwise, return to Step~\ref{step2}.
        \item If $x'_0$ is feasible, calculate its sample-average performance,
        \begin{align*}
            J_c^M(x_0')=\frac{1}{M}\sum_{j=1}^M \sum_{k=0}^N\ell_k(x_{k,j}'',\kappa(x_{k,j}')).
        \end{align*}
    \end{enumerate}
    \item \label{Step:candidateStateDef} Pick $x_0 ^\star$ to be the feasible $x_0'$ minimizing $J_c^M(\cdot)$, provided the feasible set is non-empty.
\end{enumerate}

In this algorithm, the sequences $\{x_{k,j}'\}_k$ and $\{x_{k,j}''\}_k$ are realizations of the stochastic processes generated by recursions \eqref{eq:dashrec} and \eqref{eq:ddashrec} driven by constructed independent white noise sequences $W'_j$ and $W''_j$.

The processes are functions of $\kappa$ and $x_0'$. To keep our notation compact, we omit this dependency. For a large number of realizations, $M$, the average cost $J_c^M(\cdot)$ converges, under regularity conditions \cite{resnick2019probability,doucet2000sequential},
\begin{align}
    &J_c(x_0')=\E \left (\sum_{k=0}^{N} \ell_k\left(x_k'',\kappa \left (x_{k}'\right)\right) \right ),\nonumber\\
    &=\E_{x_0''}\E_{W'}\E_{W''} \left (\sum_{k=0}^{N} \ell_k\left(x_k'',\kappa \left (x_{k}'\right)\right) \right ).\label{StateCandidacyMeasure} 
\end{align}
The expectation is factored due to the mutual independence between the elements $x_0'',W',W''$ \cite{resnick2019probability}. Further, the sample averages of indicator functions in \eqref{eq:betahat} and \eqref{eq:lambdahat} converge to probabilities of constraint satisfaction.

We shall denote the set of feasible state values $x'_0$ under this probability distribution as
\begin{align}
    \mathbb{X}_0^\epsilon&=\{ x_0' \in \Xi |\, \mathbb{P}(x_k'' \in \mathbb{X}) \geq 1-\epsilon,\,k=0,1,\hdots,N,\nonumber\\
    &\hskip -3mm \mathbb{P}(\kappa(x_k') \in \mathbb{U}) \geq 1-\epsilon,\, k=0,1,\hdots,N-1 \}. \label{bbX_0}
\end{align}

We define the \textit{candidate state} $x_0^\star$ as
\begin{equation}
    x_0^\star=\arg\min_{x_0' \in \mathbb{X}_0^\epsilon} J_c(x_0'). \label{x_0*def}
\end{equation}
The corresponding cost of $x_0^\star$ is $J_c^\star$
\begin{equation}
     J_c^\star=\min_{x_0' \in \mathbb{X}_0^\epsilon} J_c(x_0')=J_c(x_0^\star).
\end{equation}
Although these definitions are predicated on an infinite $M$, it will be shown later that $M$ of $\mathcal{O}(\log L)$ is sufficient for providing feasibility and optimality guarantees with probability/reliability at least $1-\delta$. For example, $M=135$ is used in the numerical examples in Section~\ref{Section:examples}. Where proofs are developed in the next two sections, the results are derived in terms of expectations and probabilities.

An interpretation of the central recursions \eqref{eq:dashrec} and \eqref{eq:ddashrec} is that, for each particle $x_0'\in\Xi$, \eqref{eq:dashrec} generates $M$ control sequences $\{u_{k,j}=\kappa(x'_{k,j})\}$ of a stochastically excited closed-loop system. Recursion \eqref{eq:ddashrec} then uses sample averages of the open-loop cost of these closed-loop sequences themselves then averages over the particle state density. In a sense, the search is over the performance, averaged along the horizon $N$ and over the particle density $\Xi$, of those randomized closed-loop sequences as functions of $x'_0$.

\section*{Regularity conditions for the State Selection Algorithm}
Step~(\ref{step:explore}) above creates a collection of $M$ closed-loop control sequences $\{\kappa(x_{k,j}')\}_{k=0}^{N-1}$, which are functions of $x_{0,j}'$ and the $w_{k,j}'$s; the steps after that examine the feasibility with respect to $(\mathbb X, \mathbb U)$ and performance with $J$.  In order that this stage is informative for state selection, it is important that the states and control law are suitable excited by differing $x_0'$ and by $w_k'$. Considering the system
\begin{align}
x'_{k+1}&=f(x_k',\kappa(x_k'),w_k'),\quad x_0',\label{eq:access}\\
u'_k&=\kappa(x_k'),\label{eq:obsv}
\end{align}
the algorithm requires the accessibility of \eqref{eq:access} from process noise $w_k'$, and the observability and reconstructibility of the pair \eqref{eq:access}-\eqref{eq:obsv}. This is to ensure the sensitivity of the control sequences to the initial state value and the noise process. These are conditions on the system and control law $\kappa$. Without a diverse set of control sequences for selecting $x_0'$, the non-emptiness of $\mathbb X^\epsilon_0$ cannot be assured and the minimization be effective.

As with stochastic optimal control, the closed-loop performance rests on both the control law, $\kappa$, and on the information state \cite{kumar2015stochastic} approximated by particle density $\Xi$. The propagation of the particle filter should avoid depletion issues and should properly reflect the conditional state density. We assume that it contains the conditional mean, for example, and is sufficiently encompassing to yield a rich set of feasible control sequences.

Other than these remarks, we do not delve deeper. Although, in the computed examples we point to problems with accessibility of feasible states and with the performance limitations when the particle filter is too localized. 

\section{Properties of the candidate state} \label{section: properties}
Borrowing from certainty equivalence control, where the least-squares-best state estimate, $\hat \xi_{0|0}$ the conditional mean of $\Xi$, is selected as the candidate state, we compare the cost $J_c$ for $x_0^\star$ versus that for $\hat\xi_{0|0}$.
\vskip 3mm
\begin{proposition}
Suppose that $\hat \xi_{0|0}$, the sample average of the particles in $\Xi$, is feasible, that is, $\hat \xi_{0|0}\in\mathbb{X}_{0}^\epsilon$. Then $J_c(x_0^\star)~=~J_c^\star \leq J_c(\hat \xi_{0|0})$.
\qed
\end{proposition}
This is an immediate consequence of the minimization over $x'_0$ in the final stage of the state selection algorithm.

Next, we define and compare the design cost function $J_{des}$ for comparison purposes with the cost $J_c$.
\\%\vskip 3mm
\begin{definitiondes}
Assuming full state feedback in (\ref{NonSys}), the design cost of state value $x'_0$ and its associated optimal causal control law $\kappa^\star$ are given by
    \begin{align}
    J_{des}(x_0')&=\min_{\kappa} \E_{W'} \left (\sum_{k=0}^{N} \ell_k(x_k',\kappa (x_{k}'))\right ),\nonumber\\
    &=\E_{W'} \left (\sum_{k=0}^{N} \ell_k(x_k',\kappa^* (x_{k}'))\right ),\label{OptProbDes}
    \end{align} 
\end{definitiondes}
where $\kappa^*$ is the minimizing causal control law. In the first line of (\ref{OptProbDes}), $\{x_k'\}_{k\geq0}$ is constructed using nominal control law $\kappa$, while in the second line, using $\kappa^\star$. The minimizer is assumed to exist, otherwise $\min$ is replaced by $\inf$.
\\%\vskip 3mm
\begin{proposition}\label{proposition1}
If the initial state is known and feasible, that is, $p_0$ is a point mass located at $z \in \R^{r_x}$, and $\mathbb{X}_0^\epsilon=\{z\}$, then $J_{des}(z)\leq J_c^*$.
\end{proposition}
\begin{proof}
\begin{align*}
     J_c^*&=J_c(z),\\
    &=\E_{W'}\E_{W''} \left (\sum_{k=0}^{N} \ell_k\left(x_k'',\kappa^* \left (x_{k}'\right)\right) \right ),\\
    &\geq \E_{W'}\E_{W''} \left (\sum_{k=0}^{N} \ell_k\left(x_k',\kappa^* \left (x_{k}'\right)\right) \right ),\\
    &=\E_{W'} \left (\sum_{k=0}^{N} \ell_k\left(x_k',\kappa^* \left (x_{k}'\right)\right) \right ),\\
    &=J_{des}(z),
\end{align*}
where the inequality is due to $\{x_k'\}_k$ being the optimal sequence resulting from applying the optimal control law $\kappa^*$ and starting from $x_0'=x_0''=z$. 
\end{proof}

Notice that for $J_{des}$ full state feedback is assumed for all $k \geq 0$, while the hypothesis of Proposition~\ref{proposition1} assumes full state feedback at time-0 only, in $J_c^*$. 
\vskip 3mm
\begin{proposition}\label{proposition2}
In the deterministic case, that is, $\Xi=\{z\}$ and $w_k=0$, for all $k$, if $z$ is feasible, then $J_{des}^*(z) = J_c^*$.
\end{proposition}
\begin{proof}
Having $W'=W''=0$ and $x_0'=x_0''=z$ imply $x_k'=x_k''$ for all $k$, by the definitions \eqref{eq:dashrec} and \eqref{eq:ddashrec} of these two state sequences. Thus,
\begin{align*}
    J_c^*&=J_c(z),\\
    &=\E_{W'}\E_{W''} \left (\sum_{k=0}^{T} \ell_k\left(x_k'',\kappa^* \left (x_{k}'\right)\right) \right ),\\
    &=\left (\sum_{k=0}^{N} \ell_k\left(x_k',\kappa^* \left (x_{k}'\right)\right) \right ),\\
    &=J_{des}(z),
\end{align*}
\end{proof}

Proposition~\ref{proposition2} is a restatement that for the \textit{deterministic case}, open- and closed-loop controls are equivalent and full state feedback for $k\geq 0$ is equivalent to full state feedback at $k=0$ only. In such case, the \textit{candidate state} for the optimal control $\kappa^*$ is the true state itself.

%%%%%%%%%%%%%%%%%%%%%%%%%%%%%%%%%%%%%%%%%%%%%%%%
%%%%%%%%%%%%%%%%%%%%%%%%%%%%%%%%%%%%%%%%%%%%%%%%
\section{Computational complexity of the state selection algorithm} \label{Section:NonlinearSystemsAlgorithm}
First, a lower bound on $M$ is found, for guaranteeing $\epsilon$-probabilistic feasibility with a margin $\delta$. Then, the computation time required for applying \textit{the state selection algorithm} is discussed.
\subsection{$M=\mathcal O(\log L)$}
From Section~\ref{section: ProblemFormulation}, the initial state density is given as a finite particle mass function over the particle set $\Xi=\{\xi^i|\,i=1,2,\hdots,L\}$.
\begin{align}
    p_0(\cdot)=\sum_{i=1}^{L}  \delta (\cdot - \xi^i),
\end{align}
where $\delta$ is the Dirac-delta function on $\R^{r_x}$. Since $\Xi$ is finite, so too is $\mathbb{X}_0^\epsilon$ in (\ref{bbX_0}), can be written
\begin{align*}
        \mathbb{X}_0^\epsilon&= \bigcap_{k=1}^N\{ x_0' \in \Xi |\, \mathbb{P}(x_k'' \in \mathbb{X}) \geq 1-\epsilon\} \bigcap \\
    &\hskip -3mm\bigcap_{k=0}^{N-1}\{ x_0' \in \Xi |\,\mathbb{P}(\kappa(x_k') \in \mathbb{U}) \geq 1-\epsilon\}.
\end{align*}
By construction of the random vectors $x_k''$ and $x_k'$, the condition $\{x_k'' \in \mathbb{X}\}$ depends on $x_0'$, $x_0''$, $W'$ and $W''$, while $\{\kappa(x_k') \in \mathbb{U}\}$ depends on $x_0'$ and $W'$. Hence, both are functions of $x_0'$ and $\Lambda=(x_0'',W^\prime,W^{\prime \prime})$. They can be rewritten as
\begin{align*}
    G_k(x_0',\Lambda)=\{x_k'' \in \mathbb{X}\}, \, k=1,\hdots,N,
\end{align*}
and 
\begin{align*}
    G_{k+N+1}(x_0',\Lambda)=\{\kappa(x_k') \in \mathbb{U}\},\,k=0,\hdots,N-1.
\end{align*}
Hence, more compactly,
\begin{align*}
        \mathbb{X}_0^\epsilon&= \bigcap_{k=1}^{2N}\{ x_0' \in \Xi |\, \mathbb{P}(G_k(x_0',\Lambda)) \geq 1-\epsilon\}.
\end{align*}

Statistical feasibility tolerance, $\alpha$, at Step~\ref{step:alpha} of the algorithm satisfies $\alpha \in [0,\epsilon)$ and $\Lambda^j$, for $j=1,2,\hdots,M$, at Step~\ref{step:lambda}, are independent Monte Carlo samples of $\Lambda$. Define the set
\begin{align} \label{Constraints via sampling}
        \mathbb{X}_0^{\alpha,M}&= \bigcap_{k=1}^{2N}\{ x_0' \in \Xi |\, \frac{1}{M}\sum_{j=1}^M\mathbbm{1}(G_k(x_0',\Lambda^j)) \geq 1-\alpha\}.
\end{align}
This set is a subset of $\mathbb{X}_0^\epsilon$ with probability dependent on the number of samples, $M$. For a given reliability $\delta\in(0,1)$, a sufficiently large value of $M$ can be found so that the computed sample average, $\mathbb{X}_0^{\alpha,M}$, provides a suitable approximation of $\mathbb{X}_0^\epsilon$.

The following theorem is a version of Theorem~5 of \cite{luedtke2008sample} modified to suit the problem formulation of this paper.
\vskip 3mm
\begin{theorem} \label{Theorem:NumberOfSamples}
   For any $\delta \in (0,1)$, if 
    \begin{align} \label{Equation:MlowerBound}
        M \geq \frac{1}{2(\epsilon-\alpha)^2}\log \left( \frac{L}{\delta}\right),
    \end{align}
then 
\begin{align}
    \mathbb{P}(\mathbb{X}_0^{\alpha,M} \subseteq \mathbb{X}_0^\epsilon) \geq 1-\delta.
\end{align}
\end{theorem}
\begin{lemma} \label{Hoeffdings}
    (\textbf{Hoeffding's inequality \cite{hoeffding1994probability}}). For independent random variables $Z_q,\,q=1,\hdots,\bar M$, $\mathbb{P}(Z_q \in [a_q,b_q])=1$, $a_q\leq b_q$, for all $t\geq0$
    \begin{align*}
        \mathbb{P}\left(\sum_{q=1}^{\bar M} (Z_q-\E\,Z_q) \geq t\bar M\right) \leq \exp \left(-\frac{2\bar M^2t^2}{\sum_{q=1}^{\bar M}(b_q-a_q)^2}\right)
    \end{align*}
\qed
\end{lemma}
\begin{proof}
(of Theorem~\ref{Theorem:NumberOfSamples}). Let $x \in \Xi \backslash \mathbb{X}_0^\epsilon$. Then there exists $l$, a minimizer of $\mathbb{P}(G_l(x,\Lambda))$, $l \in \{1,2,\hdots,N\}$. By definition of $x$, we have $\mathbb{P}(G_l(x,\Lambda))<1-\epsilon$. Define the random variables $Y_j$, for $j=1,2,\hdots,M$, such that $Y_j=\mathbbm{1}\{G_l(x,\Lambda^j)\}$. Consequently,
\begin{align} \label{Equation:EY_j}
    \E Y_j= \mathbb{P}\left(G_l(x,\Lambda^j)\right)<1-\epsilon.
\end{align}
Hence,
\begin{align*}
        &\mathbb{P}(x \in \mathbb{X}_0^{\alpha,M})\\
        &\leq^{\mbox{\textcircled{\tiny 1}}} \mathbb{P}\left(x\in\{ x_0' \in \Xi |\, \frac{1}{M}\sum_{j=1}^M\mathbbm{1}(G_l(x_0',\Lambda^j)) \geq 1-\alpha\}\right),\\
        &=^{\mbox{\textcircled{\tiny 2}}}\mathbb{P}\left(\frac{1}{M}\sum_{j=1}^M Y_j \geq 1-\alpha\right),\\
        &\leq^{\mbox{\textcircled{\tiny 3}}} \mathbb{P}\left(\frac{1}{M}\sum_{j=1}^M \left ( Y_j - \E Y_j \right)\geq -1+\epsilon+1-\alpha\right),\\
        &\leq \mathbb{P}\left(\sum_{j=1}^M \left ( Y_j - \E Y_j \right)\geq M(\epsilon-\alpha)\right),\\
        &\leq^{\mbox{\textcircled{\tiny 4}}} \exp \left(-2 M(\epsilon-\alpha)^2\right),
\end{align*}
where the numbered inequalities follow from: {\textcircled{\tiny 1}}- probability of an intersection of events underbounds that of any single event; {\textcircled{\tiny 2}}, {\textcircled{\tiny 3}}- the definition of $Y_j$ in \eqref{Equation:EY_j}; {\textcircled{\tiny 4}}- Hoeffding's inequality, Lemma~\ref{Hoeffdings}, with $a_j=0,\,b_j=1$ for all $j \in \{1,2,\hdots,M\}$. Therefore,
\begin{align*}
    \mathbb{P}\left ( \mathbb{X}_0^{\alpha,M} \not\subseteq \mathbb{X}_0^\epsilon\right)&= \mathbb{P}\left \{ \textrm{there exists } x \in \mathbb{X}_0^{\alpha,M} \textit{ s.t. } x \not \in \mathbb{X}_0^\epsilon \right \},\\
    &\leq \sum_{x \in \Xi \backslash \mathbb{X}_0^\epsilon} \mathbb{P}(x \in \mathbb{X}_0^{\alpha,M}),\\
    &\leq L\exp \left(-2 M(\epsilon-\alpha)^2\right).
\end{align*}
If the left-hand-side is to be $\delta \in (0,1)$, then, taking the $\log(\cdot)$ of both sides, results in (\ref{Equation:MlowerBound}). 
\end{proof}
If $x_0'\in \mathbb{X}_0^{\alpha,M}$, then $x \in \mathbb{X}_0^\epsilon$ with probability at least $1 - \delta$. A similar result providing a lower bound of $M$ of $\mathcal{O}(\log L)$ can be found for achieving optimality: $x_0'$ in Step~\ref{Step:candidateStateDef} of the algorithm satisfies \eqref{x_0*def} with probability close to one, using an extension of Corollary~6 of \cite{luedtke2008sample}.

The approximations inherent in selecting parameters $\alpha\in(0,\epsilon)$ and $\delta\in(0,1)$, together with the imprecision in the inequality \eqref{Equation:MlowerBound}, are the source of the inbuilt conservatism of the quantifications of the State Selection Algorithm. This will become evident in the achieved constraint violation rates in Section~\ref{Section:examples} where computational examples are conducted with sample repetition value $M=135$ and particle count $L=400$.

\subsection{Computation time}
The State Selection Algorithm has computational complexity $\mathcal{O}(N L\log L)$, per time-step. It is parallelizable in: its second step; across $i$ in the choice of $x_0'$; and also across $j$ over the samples. Using a graphics processing unit (GPU) would further decrease the computation time required. 

As the dimension, $r_x$, of the state increases, the complexity of this approach increases via the requisite (perhaps dramatic) increase in the number of samples, $L$, of the particle filter. Once $L$ is fixed, however, the calculations for $M$ given $\epsilon$, $\alpha$ and $\delta$ abide.
%Algorithm was here, check for discontinuities

\section{Constrained stochastic linear systems with quadratic cost}\label{Section:LinearPolyhedron}
In this section, we show that under the assumptions of: linear dynamics, linear state feedback control law, quadratic stage costs, and polyhedral state and input constraints, the candidate state is the solution of a quadratic program over $\R^{r_x}$ based on the second-order moments of the densities. This obviates the need for: calculation of $M$ sample sequences and their sample averages of constraint violations and costs, and the inclusion of margin parameters $\alpha$ and $\delta$ earlier. We assume that the filtered conditional state density is described by its first two moments, which, in turn, might be provided by a Kalman filter or particle filter.

At time $k=0$, we have:
\begin{enumerate}[label=\Roman*.]
\item Discrete-time linear state dynamics
\begin{align}
x_{k+1} & = F x_{k} + G u_k+    w_k.\label{LinSys}
\end{align}
\item $\{w_k\}$ is i.i.d., $\E w_k=0$ and $\Cov(w_k)=\Sigma_w$, for all $k$.
\item Minimal constraint violation probability $\epsilon\in(0,1)$ for polyhedral state and input constraint sets,
\begin{equation} \label{Constraints1}
\begin{aligned}
    &\mathbb{X}=\{x_k \in \R^{r_x}| T x_k \leq \bar x \},\, \bar x \in \R^{t},\\
    &\mathbb{U}=\{u_k \in \R^{r_u}| S u_k \leq \bar u \},\,  \bar u \in \R^{m},
\end{aligned}
\end{equation}
where $T \in \R^{t \times r_x}$ and $S \in \R^{m \times r_u}$ have full row rank.
\item\label{step:sigma} The first two moments, $\E x_0=\hat x_0$ and $\Cov(x_0)= \Sigma_0,$ of the density of the initial state $x_0$, which is independent from $w_k$ for all $k$.
\item Linear full-state-feedback control law $u_k=K x_k,$  which is yet neither assumed to be recursively feasible with respect to $\mathbb X$ and $\mathbb U$ nor optimal with respect to a cost, nor indeed stabilizing. 
\item A finite-horizon quadratic trajectory cost function $J$ 
\begin{align*}
    J = x_k^TQ_Nx_k + \sum_{k=0}^{N-1}\left[x_k^TQx_k+u_k^TRu_k\right].
\end{align*}
\end{enumerate}

We derive a variant of the State Selection Algorithm for this case, where expectations of states, costs, constraint violation are directly characterized without the need for sample averages of simulations. The net result is a quadratic program to determine the selected state $x_0^\star$. This is tantamount to operating the State Selection Algorithm with very large $M$ and without the attendant computational demands.

We shall see in the numerical examples later that the state selection algorithm for these constrained linear quadratic problems returns the conditional mean when the constraints are inactive.

The state sequences in \eqref{eq:dashrec} and \eqref{eq:ddashrec} become
\begin{align}
    x_{k+1}'&=(F+GK)x_k'+w_k',\, x_0' \in \R^{r_x}, \label{x_k'defLinear} \\
    x_{k+1}''&=F x_k''+G K x_k'+w_k'',\, x_0'' \sim p_0.\label{x_k''defLinear}
\end{align}
These recursions from the State Selection Algorithm yield:
\begin{itemize}
\item $x'_k$ affine in $x'_0$ and $w'_{k-j}$,
\item $x''_k$ affine in $x'_0$, $x''_0$, $w'_{k-j}$ and $w''_{k-j}$.
\end{itemize}
In turn, this implies that the expected cost function, $J_c$ of \eqref{StateCandidacyMeasure}, is quadratic in $x'_0$. The analysis of the constraints is more work but results in new, but more conservative, polyhedral constraints on $x'_0$.

\subsection{The cost function $J_c$}
For $x_0' \in \mathbb{X}_0^\epsilon$, the cost function $J_c$ is given by \eqref{StateCandidacyMeasure}.
\begin{align}
    J_c(x_0')&=\E_{x_0''}\E_{W'}\E_{W''} \bigg (\sum_{k=0}^{N-1} \Big[ (x_k'')^T Q x_k'' \nonumber\\
    &\hskip 7mm+(x_k')^TK^T R K x_k'\Big] + (x_N'')^TQ_Nx_N'' \bigg ), \label{StateCandidacyMeasureLinear}\\
%\end{align}
%which can be equivalently expressed as
%\begin{align}
    &=\E_{x_0''}\E_{W'}\E_{W''} \bigg (\sum_{k=0}^{N-1} \Big[ \trace (Q x_k''(x_k'')^T) \nonumber\\
    &\hskip -5mm+\trace(K^T R K x_k'(x_k')^T)\Big] + \trace(Q_Nx_N''(x_N'')^T) \bigg ). \label{StateCandidacyMeasureLinearTrace}
\end{align}
From (\ref{x_k'defLinear}) and (\ref{x_k''defLinear}), 
\begin{align}
    x_k'&=F_K^k x_0'+\sum_{p=0}^{k-1}F_K^p w'_{k-p-1}, \label{linear'}\\
    x_k''&=F^k x_0''+\sum_{j=0}^{k-1}F^j w''_{k-1-j} \nonumber\\
    &+ \Psi_{k-1} x_0'
    +\sum_{h=0}^{k-1}F^h GK \sum_{p=0}^{k-h-2}F_K^p  w'_{k-h-2-p}, \label{linear''}
\end{align}
where $F_K=(F+GK)$ and $\Psi_{k-1}=\sum_{h=0}^{k-1}F^h G K F_K^{k-h-1}$. That is, these state sequences are affine functions of $x_0''$ and $x_0'$. Hence, the cost $J_c$ is quadratic in $x_0'$ and $x_0''$. 

Denoting mean values,  $\hat x_k'=\E_{W'}x_k'$ and $\hat x_k''=\E_{x_0''}\E_{W'}\E_{W''}x_k''$, we have
\begin{align*}
    \hat x_k'=F_K^k x_0', \quad
    \hat x_k''=F^k \hat x_0''+ \Psi_{k-1} x_0'. \label{means}
\end{align*}

Hence, the errors $\tilde x_k'=x_k'-\hat x_k'$ and $\tilde x_k''=x_k''-\hat x_k''$ are zero by dint of $w'_k$ and $w''_k$ being zero mean.

Using independence assumptions, and ignoring additive constants,
\begin{align}
    J_c(x_0')&=(x_0')^T \mathcal{A}_1 x_0'+(\hat x_0'')^T \mathcal{A}_2 x_0'.
\end{align}
The formul\ae\ for and derivation of $\mathcal{A}_1$ and $\mathcal{A}_2$ can be found in the Appendix.

\vskip 3mm
\begin{corollary}
    Suppose that $K_k$ is the optimal time-variant feedback gain of the unconstrained N-horizon LQ problem with $Q$, $R$ and $Q_N$. Then the candidate state is the conditional mean.\qed
\end{corollary}
This result follows directly from the discrete-time equivalent of the argument provided in \cite[p.~221]{Anderson&Moore:89} and the monotonic dependence of the control performance on the trace of the state estimate covariance in the unconstrained case.

\subsection{The constraints}
The state and input sequences $\{x_k''\}$, $\{u_k\}=\{Kx_k'\} $ are both stochastic by construction. The prediction covariance, $\Sigma'_k$, of $x_k'$ can be described by
\begin{align}\label{covx'}
    \Sigma_{k+1}'=(F+GK)\Sigma_k'(F+GK)+\Sigma_w,\, \Sigma'_0=0,
\end{align}
where the initial covariance is zero because $x_0'$ is deterministic. For $\Sigma''_k$, the covariance of $x_k''$, we have 
\begin{align}\label{covx''}
    \Sigma''_{k+1}=F \Sigma_k'' F^T+ G K \Sigma_k' K^T G^T + \Sigma_w,\, \Sigma_0''=\Sigma_{0}.
\end{align}
The initial covariance $\Sigma_0$ is provided at Item~\ref{step:sigma} in the linear problem statement.

The input sequence covariance follows from (\ref{covx'}).
\begin{equation}
    \Cov(u_k)=\Cov(Kx_k')=K\Sigma_k'K^T.
\end{equation}

Next, we show how the probabilistic constraints in (\ref{Constraints1}) can be transformed into deterministic linear constraints in terms of the state sequence means (\ref{means}) and covariances (\ref{covx'}), (\ref{covx''}).

\vskip 3mm
\begin{proposition} \label{Prop:probtolin}
The probabilistic polyhedral state constraints in (\ref{Constraints1}) are satisfied if the following deterministic polyhedral constraints are satisfied
\begin{align}
    T\hat x_k'' \leq \bar x - \sqrt{\frac{t-\epsilon}{\epsilon}}\sqrt{\textrm{diag}(T \Sigma_{k} T^T)}, \label{constraintsProp}
\end{align}
where: $t$ is the number of rows of $T_k$, the function $\textrm{diag}(\cdot)$ returns the diagonal of a square matrix as a column vector and the second square root is elementwise.
\end{proposition}
\vskip 3mm
\begin{lemma} \label{Cantelli'sLemma}
\textbf{(Cantelli's inequality)} For a scalar random variable $\gamma$ with mean $\hat \gamma$ and variance $\Gamma$,
\begin{equation} \label{Cantelli's}
    \mathbb{P}(\gamma-\hat \gamma \geq \eta) \leq \frac{\Gamma}{\Gamma+\eta^2}, \, \eta \geq 0.
\end{equation}
\qed
\end{lemma}
\vskip 3mm
\begin{lemma} \label{lemmaSets}
For $j=1,\hdots,t$, let $T(j)$ be the $j^{\textit{th}}$ row of $T$ and $\bar x (j)$ be the $j^{\textit{th}}$ element of $\bar x$. The probabilistic constraints
\begin{align}
    \mathbb{P}\Big (T(j)x_k'' \leq \bar x(j)\Big) \leq 1- \frac{\epsilon}{t}, \label{lemma2Ineq}
\end{align}
are satisfied if the following linear inequality holds
\begin{equation} \label{Cantelli'sExt}
   T(j) \hat x_k'' \leq \bar x(j) - \sqrt{\frac{t-\epsilon}{\epsilon}}\sqrt{T(j) \Sigma_k'' T(j)^T}.
\end{equation}
\begin{proof}
Analogous to the work in \cite{farina2013probabilistic}, suppose there exists $\rho \geq 0$ such that
\begin{align} \label{deltaEquation}
T(j) \hat{x}'' \leq \bar x - \rho,
\end{align}
hence the condition $T(j)x''\geq T(j) \hat x''+\rho$ is implied by $T(j)x''\geq \bar x(j)$, so that,
\begin{align*}
    \mathbb{P}\Big(T(j)x''\geq \bar x\Big) &\leq \mathbb{P}\Big(T(j)x''\geq T(j) \hat x''+\rho\Big),\\
    &=\mathbb{P}\Big(T(j)x''-T(j) \hat x''\geq \rho\Big),\\
    &\leq \frac{T(j) \Sigma'' T(j)^T}{T(j) \Sigma'' T(j)^T+\rho^2},
\end{align*}
where the last inequality follows from Cantelli's inequality (\ref{Cantelli's}). If $\rho$ is sufficiently large that the last term is upper bounded by $\epsilon/t$,
\begin{align*}
    \frac{T(j) \Sigma'' T(j)^T}{T(j) \Sigma'' T(j)^T+\rho^2} \leq \frac{\epsilon}{t},
\end{align*}
or equivalently
\begin{align*}
    \rho \geq \sqrt{\frac{t-\epsilon}{\epsilon}}\sqrt{T(j)\Sigma''T(j)^T}.
\end{align*}
This lower bound of $\rho$, when used in (\ref{deltaEquation}), yields (\ref{Cantelli'sExt}).
\end{proof}
\end{lemma}
\vskip 3mm
\begin{lemma} \label{intersectionsLemma}
Let $(\Omega, \mathcal{B}, \mathbb{P})$ be a probability space and $E_i \in \mathcal{B}$ for $i=1,\hdots,n$. If $\mathbb{P}(E_i) \geq 1-\epsilon/n$, for all $i=1,\hdots,n$, then $\mathbb{P}(\bigcap_{i=1}^n E_i)\geq 1-\epsilon$. \qed
\end{lemma}
\begin{proof}
\textbf{(of Proposition~\ref{Prop:probtolin})}.  Notice that the state constraint sets in (\ref{Constraints1}) can be written as intersection of sets
\begin{align*}
    \mathbb{X}_k&=\{x_k \in \R^{r_x}| T x_k \leq \bar x_k \},\\
    &=\bigcap_{j=1}^t \{x_k \in \R^{r_x}| T(j) x_k \leq \bar x_k(j) \},
\end{align*}
since all rows are to be enforced simultaneously \cite{luedtke2008sample}. By Lemma~\ref{intersectionsLemma}, $\mathbb{P}(x_k \in \mathbb{X}_k) \geq 1-\epsilon$ is implied by $\mathbb{P}(\{x_k \in \R^{r_x}| T(j) x_k \leq \bar x_k(j) \}) \geq 1-\epsilon/t$. The latter is implied by (\ref{Cantelli'sExt}) in Lemma~\ref{lemmaSets}. Stacking the inequalities in (\ref{Cantelli'sExt}) for all of the $t$ rows of $T$, we get (\ref{constraintsProp}). Notice that with an increase in the number of rows $t$, the constraints become tighter and the approximation more conservative.
\end{proof}
Using parallel arguments, results analogous to those above hold for the probabilistic input constraints. Finally, the probabilistic constraints (\ref{Constraints1}) remain polyhedral in $x_0'$ as follows with $K$ being the nominal state feedback gain and $S$ the matrix in the control constraint \eqref{Constraints1}.
\begin{equation} \label{allConstraints}
    \begin{aligned}
    T \hat x_k'' &\leq \bar x - \sqrt{\frac{t-\epsilon}{\epsilon}}\sqrt{\textit{diag}(T \Sigma_k'' T^T)},\\
    &\hskip -3mm\textit{for all } k=1,2,\hdots,N,\\
    S K\hat x_k' &\leq \bar u - \sqrt{\frac{m-\epsilon}{\epsilon}}\sqrt{\textit{diag}(S K\Sigma_k' K^T S^T)},\\
    &\hskip -3mm\textit{for all } k=1,2,\hdots,N-1,\\
    S K\hat x_0' &\leq \bar u,
    \end{aligned}
\end{equation}
where $\hat x_k'$ and $\hat x_k''$ are solely functions of $x_0'$, as in (\ref{means}). The covariance matrices $\Sigma_k'$ and $\Sigma_k''$ can be computed offline and are independent from $x_0'$. The last inequality is due to the fact that $\Sigma_0'=0$.

We follow \cite{yan2005incorporating} in the usage of the `closed-loop covariance' in (\ref{allConstraints}). That is, we replace the covariance matrices $\Sigma_k'$ and $\Sigma_k''$ by their one step ahead predictions 
\begin{equation}
    \Sigma_1'= \Sigma_w,\quad
    \Sigma_1''=F\Sigma_0''F^T+\Sigma_w, \label{eqn:Sigmas}
\end{equation}
where $\Sigma_0''=\Sigma_0$ is given in (\ref{NonSys}). This accounts for the fact that at the next time step, a new measurement of the system will be available. Thus, this relaxes the constraints and avoids the unbounded growth in $k$ of $\Sigma_k''$ when $F$ is unstable.

The minimization problem, to find $x_0^\star$, in a compact form is a quadratic program in $x_0'$.
\begin{equation} \label{LinearCaseOpt}
    \begin{aligned}
    \min_{x_0'}& \quad (x_0')^T \mathcal{A}_1 x_0'+(\hat x_0'')^T \mathcal{A}_2 x_0',\\
    &\hskip -3mm\textit{subject to}\\
     T \Psi_{k-1}  x_0' &\leq \bar x - \sqrt{\frac{t-\epsilon}{\epsilon}}\sqrt{\textit{diag}(T \Sigma_1'' T^T)}-T F^k \hat x_0'',\\
    &\hskip -3mm\textit{for all } k=1,2,\hdots,N\\
    S K F_K^k x_0' &\leq \bar u - \sqrt{\frac{m-\epsilon}{\epsilon}}\sqrt{\textit{diag}(S K\Sigma_w K^T S^T)},\\
    &\hskip -3mm\textit{for all } k=1,2,\hdots,N-1\\
    S K x_0' &\leq \bar u,
    \end{aligned}
\end{equation}

For linear quadratic stochastic problems with polyhedral constraints, the reduction of the State Selection Algorithm to a quadratic program for $x_0'$ admits an appreciation of the methodology absent the sampling and the sample-average convergence requirements. In turn, this allows comparison with known solutions from the deterministic case \cite{BemporadMorariDuaPistikopoulosAutom2002} and conditional mean state estimates from Kalman filtering.

\section{Numerical examples} \label{Section:examples}
In this section, we present computational examples in which the State Selection Algorithm of Section~\ref{Section:NonlinearSystemsAlgorithm} is applied. We also present an example that belongs to the case of linear systems with quadratic cost and polyhedral constraints; amenable to the quadratic program methods presented in Section~\ref{Section:LinearPolyhedron}. The computational burden of the algorithm is also evaluated, including the propagation of the bootstrap particle filter.

\subsection{Nonlinear system}
We follow the problem formulation schema from Section~\ref{section: ProblemFormulation}.
\begin{enumerate}[label=\Roman*.]
\item The state dynamics are described by 
\begin{align*}
z_{k+1}&=0.9z_k+0.2h_k+w_k^1,\\
h_{k+1}&=-0.15z_k+0.9h_k+0.05z_kh_k+u_k+w_k^2,\\
y_k&=z_{k}+v_k.
\end{align*}
%\vskip -3mm
Here $x_k=(\begin{array}{ll}z_k&h_k\end{array})^T\in\mathbb R^2$ is the state vector, $u_k$ the scalar control, $y_k$ the scalar measurement, and $w_k=(\begin{array}{ll}w^1_k&w^2_k\end{array})^T\in\mathbb R^2$ the process noise, and $v_k$ is the scalar measurement noise. 
\item The noises $w_k\sim\mathcal N(0_2,0.3\mathbb I_2)$ and $v_k\sim\mathcal N(0,0.3)$\footnote{$\mathcal{N}(\mu,\Sigma)$ denotes a Gaussian density with mean vector $\mu$ and covariance matrix $\Sigma$.}. 
\item The input constraint set is $\mathbb{U}=[-3,3]$. The state constraint set, $\mathbb{X}$, is the complement of $\mathcal{L}$ in $\R^2$, where $\mathcal{L}=[3,5]\times[-4,2] \cup [-2,5]\times [-7,-4]$. This L-shaped set to be avoided is depicted in the figures below. The constraint violation rate is $\epsilon=0.3$.
\item The $x_0$ state density is provided by a collection of $L=400$ particles in $\mathbb R^2$.
\item We consider two successive nominal controllers: \begin{enumerate}
\item the stabilizing feedback-linearizing control law $u_k=\kappa_1(x_k)=-0.05z_kh_k$; and then,
\item a feasible optimal controller, $\kappa_2(x_k)$, to be detailed shortly.
\end{enumerate}
\item The running cost is $\ell_k(x_k,u_k)=x_k^Tx_k+u_k^2$. The horizon $N=6$.
\end{enumerate}
For the State Selection Algorithm, choose $\alpha=.1$, $\delta=.01$. Theorem~\ref{Theorem:NumberOfSamples} then admits $M=135$. The following sequence is then conducted starting from $\Xi$ being $L=400$ particles sampled from $\mathcal{N}((\begin{array}{ll}7.5&-7.5\end{array})^T,.5\mathbb{I}_2)$.
\begin{enumerate}[label=\roman*)]
\item The selected state, $x_0^\star$, is applied in the control $u_k=\kappa(x_0^\star)$,
\item The output $y_{k+1}$ is measured,
\item A bootstrap particle filter computes an updated set, $\Xi,$ of $L$ filtered particles for $x_{k+1}$,
\item The state selection is re-performed.
\end{enumerate}
With this iteration, the State Selection Algorithm, its attendant particle filter, and the nominal control law can be evaluated jointly for their control performance and constraint handling. Since the paper purports to study state estimation for control, this is a critical evaluation.

To visualize the open-loop dynamics better, Figure~\ref{fig:streamLines} displays the streamlines of the state dynamics with zero input and zero state disturbances.
\begin{figure}[h]
\centering 
\includegraphics{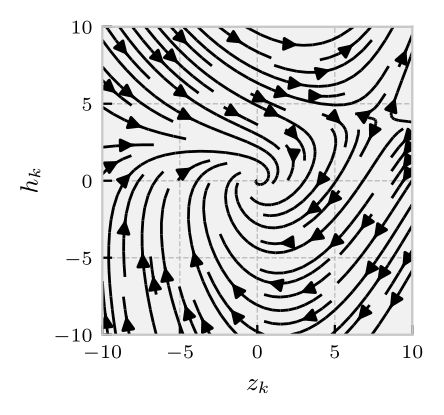} \caption{The streamlines of the nonlinear system state with $u_k$ and $w_k$ set to zero.\label{fig:streamLines}}
\end{figure}

\subsubsection{Example 1: stabilizing controller}\label{subsubsection:example1}

We conduct two comparative simulations of the controlled nonlinear system: one with the State Selection Algorithm, as outlined above, i.e. $u_k=\kappa_1(x_0^\star)$; and the other with what might be termed the certainty equivalence controller, $u_k=\kappa_1(\hat x_{0|0})$ with $\hat x_{0|0}$ being the corresponding particle filter conditional mean. In each case the particle filter evolves according to the respective measured output, which in turn depends on the applied control. Figure~\ref{fig:L_example_SS} displays the controlled state filtered particle density with the State Selection Algorithm feedback. Figure~\ref{fig:L_example_CM} shows the corresponding conditional mean feedback case. Figure~\ref{fig:violationPercentStabilizing} shows the percentage of the particles, at each time step, which violate the state constraints and land inside the set $\mathcal{L}$. This shows that State Selection Algorithm enforcing more caution. 
\begin{figure}[h]
\centering 
\includegraphics{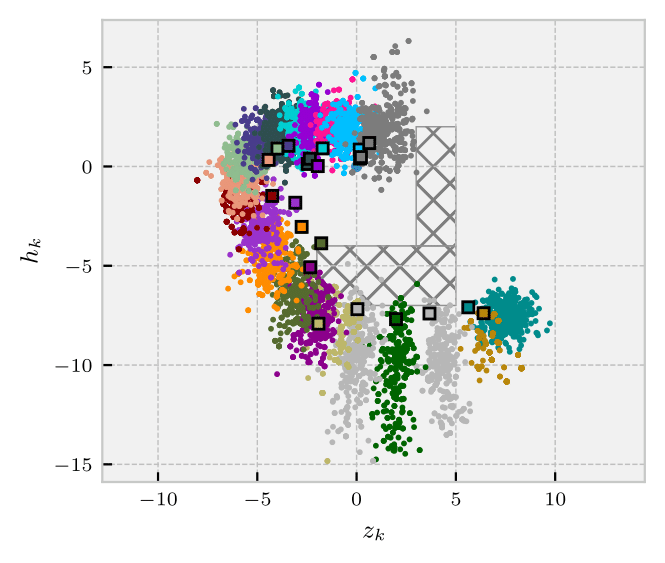} \caption{\textbf{State Selection Algorithm:} The evolution of the particle filtered density with $u_k=\kappa_1(x_0^\star)$, that is, the \textit{candidate state} $x_0^\star$ being used by the controller at each time step. The black squares indicate the location of the selected states from the particle densities. The cross-hatched object is the complement of the state constraint set, $\mathbb X$.\label{fig:L_example_SS}}
\end{figure}
\begin{figure}[h]
\centering 
\includegraphics{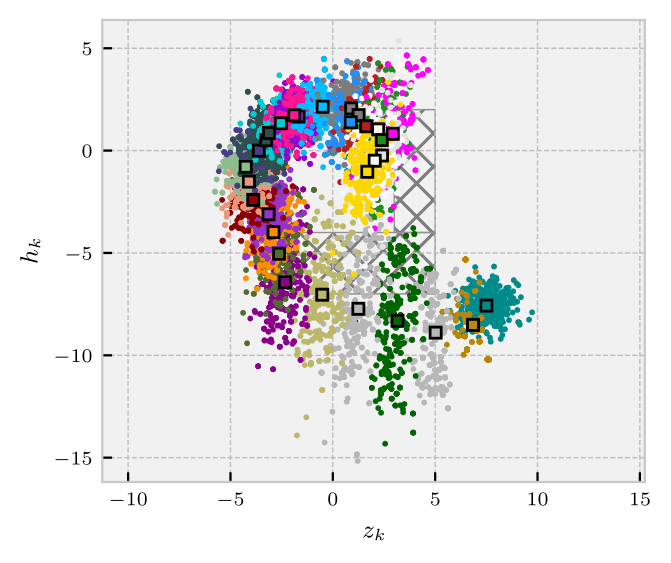} \caption{\textbf{Conditional mean:} The evolution of the particle filtered density with $u_k=\kappa_1(\hat x_{0|0})$, where $\hat x_{0|0}$ is the conditional mean of the particle filtered density indicated by the black square at each time.\label{fig:L_example_CM}}
\end{figure}

We make the following observations concerning these two simulations.
\begin{itemize}
\item The closed-loop particle densities evolve differently because the control signals, and therefore states and measurements, differ.
\item The State Selection Algorithm avoids state constraint violation, perhaps too conservatively, while the conditional mean control violates the state constraint requirements, as measured using the particle density.
\item It is evident that the selected state enforces caution into the subsequent particle density by choosing $x_0^\star$ close to the constraint boundary. This is a property dependent on the nominal control law $\kappa_1(\cdot)$.
\item The state selection simulation stopped at time-step $k=22$ because the feasible set of states in the particle density, $\mathbb{X}_0^{\alpha,M}$, was empty. That is, no state choice $x'_0$ yielded a feasible solution. Again, this is a property stemming from the lack of reachability of the $\kappa_1$-controlled dynamic system close to the origin.
\end{itemize}

\begin{figure}[h]
\centering 
\includegraphics{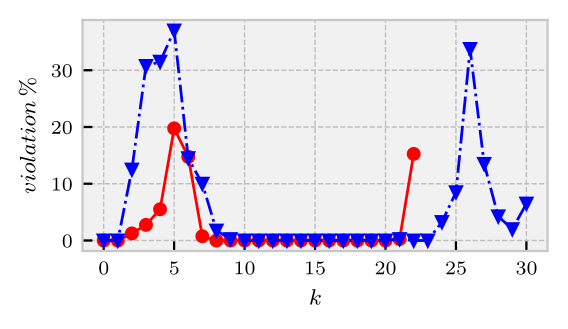} \caption{The blue triangles indicate the state constraint violation rate for the nonlinear system with control $u_k=\kappa_1(\hat x_{0|0})$, that is, using the current conditional mean of the particle density. The red circles indicate the same rate for the control $u_k=\kappa_1(x_0^\star)$ based on the state selection algorithm.\label{fig:violationPercentStabilizing}}
\end{figure}
The simulations were conducted using \python, an uncompiled interpretive program, on an M1-chip 2021 MacBook Pro with 16.00 GB of RAM. The average running time for each time step of the simulation in Figure~\ref{fig:L_example_SS}, comprising the State Selection Algorithm (the dominant load) and the particle filter, is about $4.9$ seconds. The algorithm is completely parallelizable, over the state choice $x_0'$ and over the corresponding samples. Although this was not implemented here, it can offer a potential improvement in computation time.

\subsubsection{Example 2: feasible optimal controller}

We next re-conduct the previous experiment with state feedback controller $\kappa_2(\cdot)$: a ($\mathbb U,\mathbb X)$-feasible, infinite-horizon discounted-cost, optimal controller computed by sampling and a value iteration. Its construction is detailed in the Appendix.

\begin{figure}[h]
\centering 
\includegraphics{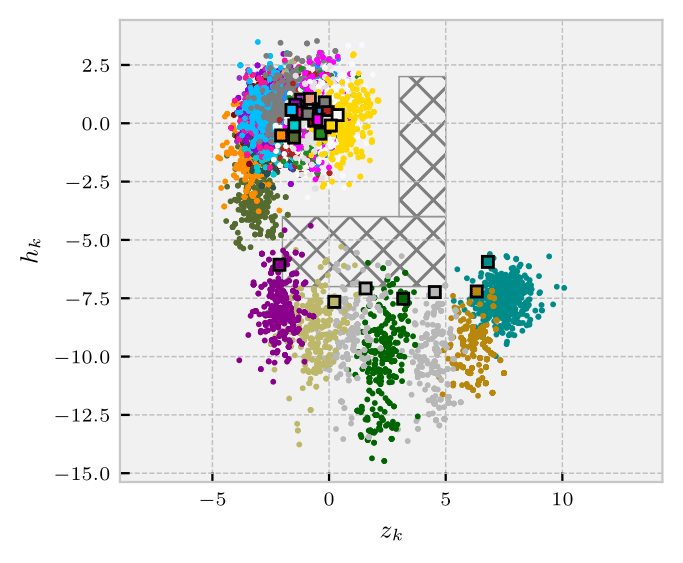} \caption{\textbf{State Selection Algorithm:} The evolution of the particle filtered density with $u_k=\kappa_2(x_0^\star)$. That is, the \textit{candidate state} $x_0^\star$, indicated by the black squares, being used by the controller at each current time step. The black squares indicate the selected state.\label{fig:L_example2_SS}}
\end{figure}
\begin{figure}[h]
\centering 
\includegraphics{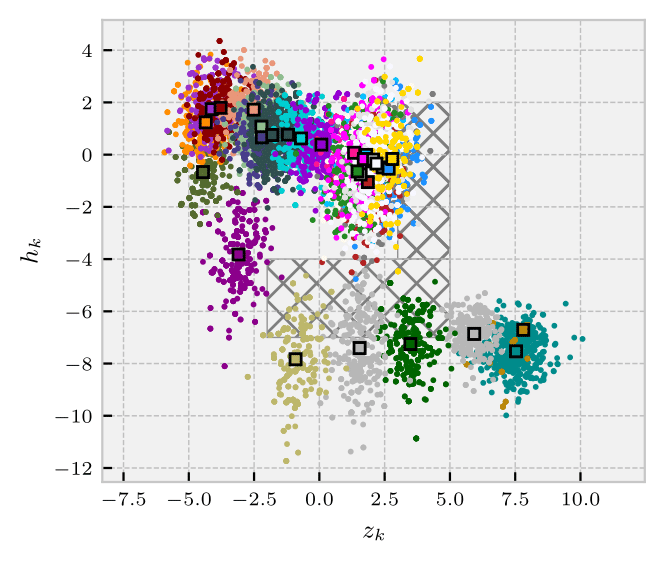} \caption{\textbf{Conditional mean:} The evolution of the particle filtered density with $u_k=\kappa_2(\hat x_{0|0})$, that is, the particle filter conditional mean being used by the controller at each current time step.\label{fig:L_example2_CM}}
\end{figure}

For control law $\kappa_2(\cdot)$, two closed-loop simulations were conducted. Figure~\ref{fig:L_example2_SS} displays the result of using the State Selection Algorithm's $x_0^\star$ in the controller. Figure~\ref{fig:L_example2_CM} shows the corresponding behavior when the particle filters' conditional mean, $\hat x_{0|0}$, is used. Figure~\ref{fig:violationPercent} shows the percentage of the particles, at each time step, which violate the state constraints. The value of $\epsilon$ is 0.3.

\begin{figure}[h]
\centering 
\includegraphics{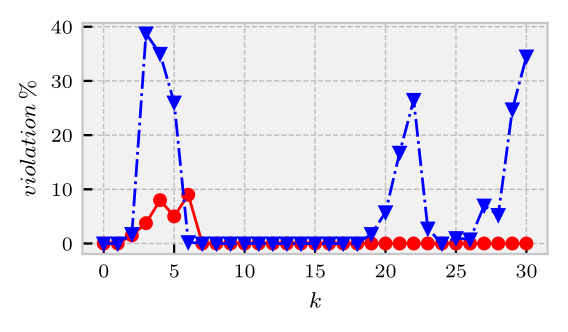} \caption{The blue triangles indicate the state constraint violation rate for the nonlinear system with control $u_k=\kappa_2(\hat x_{0|0})$, that is, using the conditional mean of the particle density. The red circles indicate the same rate for the control $u_k=\kappa_2(x_0^\star)$ based on the state selection algorithm. The value of $\epsilon$ is 30\%.\label{fig:violationPercent}}
\end{figure}

\vskip 3mm
\begin{remark}
The information state/particle density $\Xi$, as discussed in Section~\ref{section: ProblemFormulation}, plays a central role in the State Selection Algorithm; it diversifies the closed-loop control sequences $\{\kappa(x_{k,j}')\}_{k=0}^N$, since $x_0'$ is chosen from the particles in $\Xi$. Therefore, the requirements of a particle filter in the State Selection Algorithm can be categorized in two parts. Firstly, the particle filter has to be a sufficiently accurate approximation to the Bayesian filter. This is a foundational assumption upon which the algorithm is built. Secondly, particle depletion must be avoided in order that particle variability is preserved.

In the above examples, we used a measurement disturbance, $v_k$, of variance $0.3$ and state disturbance, $w_k$, of covariance $0.3\mathbb I_2$. This facilitated retaining diversity in the particle filtered density as shown in the figures. Selecting the variance of the measurement disturbance to be $0.1$ for instance, with the number of particles fixed to $400$, resulted in particle depletion and hence to infeasibility issues with the State Selection Algorithm and furthermore to a poor representation of the Bayesian filter density.

Converse to but alongside the requirements of the particle filter rest those for the control law, $\kappa$, and the state disturbance $w_k$. This subject is broached in Section~\ref{section: ProblemFormulation} as a regularity condition. For the State Selection Algorithm the accessibility of the state and control signals from $w_k$ affects the closed-loop feasibility set, $\mathbb X_0^\epsilon$, for the algorithm. In the nonlinear example with feedback linearizing controller, $\kappa_1$, the feasible set is empty after 22 steps -- this changes with each run. This is a result of a diminished control gain near the origin resulting in poor excitation.
\end{remark}

\subsubsection{Example 3: Linear dynamics with polyhedral constraints}
In this example, a standard DC-DC converter regulation problem is considered. It is a benchmark in the stochastic MPC literature and used in \cite{schluter2022stochastic,cannon2010stochastic,lorenzen2016constraint}. 
\begin{itemize}
\item The dynamics are described by
\begin{align*}
    \begin{pmatrix}
        x_{k+1}^1\\
        x_{k+1}^2
    \end{pmatrix}= x_{k+1}&=Fx_k+Gu_k+w_k,\\
    y_k&=Hx_k+v_k,
\end{align*}
where
\begin{align*}
    F&= \begin{bmatrix}
        1 &0.0075\\
        -0.143 &0.996
    \end{bmatrix}, \quad 
    G= \begin{bmatrix}
        4.798\\
        0.115
    \end{bmatrix},\\
    H&= \mathbb{I}_{2 \times 2}.
\end{align*}
    \item Noise signals $w_k$ and $v_k$ are white with zero means, and independent from each other and from $x_0$.
    \begin{align*}
&\Cov(w_k)=\Sigma_w=0.1 \mathbb{I}_{2\times 2},\,\Cov(v_k)=\textrm{diag}(0.5,0.4),\\
&\E x_0 = (0.6455,1.3751)^T,\, \Cov(x_0)=\Sigma_0=0.1 \mathbb{I}_{2 \times 2}.
\end{align*}
We take the densities to be Gaussian in the simulations.

\item The probabilistic constraints are 
\begin{equation*}
    \mathbb{P}(x^1_k \leq 2) \geq 1-\epsilon, \textrm{ for all $k=1,\hdots,N$}
\end{equation*}
where $\epsilon=10\%$, differently from the 40\%\ of \cite{schluter2022stochastic}. According to Proposition~\ref{Prop:probtolin} and (\ref{LinearCaseOpt}), this is implied by
\begin{align*}
    T \Psi_{k-1} \hat x_0 &\leq  \bar x - \sqrt{\frac{t-\epsilon}{\epsilon}}\sqrt{\textrm{diag}(T \Sigma_{1}'' T^T)}-TF^k\hat x_0'',
\end{align*}
for all $k=1,\hdots,N$. Where $T=(1,0)$, $\bar x=2$, $t=1$, and $\Sigma_1''$ is the prediction covariance and defined in (\ref{eqn:Sigmas}).

\item The linear state-feedback controller $u_k=Kx_k$ is chosen to be the infinite horizon LQR controller with the weighting matrices, similar to those in \cite{schluter2022stochastic}, $Q=\textrm{diag}(1,10)$ and $R=10$%\footnote{The convention in this paper is $u_k=+Kx_k$.},
\begin{equation*}
    K= \begin{bmatrix}
        -0.2409   & 0.3930
    \end{bmatrix},
\end{equation*}
these weighting matrices are also used in forming the optimization problem (\ref{LinearCaseOpt}), and $Q_N=Q$.

\item The prediction horizon, for (\ref{LinearCaseOpt}), is chosen to be $N=8$. The associated matrices $\mathcal{A}_1$ and $\mathcal{A}_2$ are
\begin{align*}
\mathcal{A}_1= \begin{bmatrix}
    47.23 & -43.76\\
    -43.76 & 45.51
\end{bmatrix},\quad \mathcal{A}_2=
\begin{bmatrix}
    -93.98 & 87.18\\
 85.33 & -89.45
\end{bmatrix}
 ,
\end{align*}
\end{itemize}

The quadratic program \eqref{LinearCaseOpt} is solved at each time to find the candidate state $x_0^\star$, which is used by the linear state-feedback controller. The state conditional mean and covariance are updated using the Kalman filter, which is the least-squares optimal unbiased estimator still. The result of $100$ closed-loop iterations is shown in Figure~\ref{fig:Linear_SS}. The corresponding simulation using the state conditional mean from the Kalman filter with the linear controller is shown in Figure~\ref{fig:Linear_CM}. 

Notice that when $x_k^2$ is below a certain line, the optimization problem returns, approximately, the conditional mean as the candidate state. However, above that line, the candidate state differs. This is the effect of the constraint, present in the state section process but not in LQR or the Kalman filter. 
\begin{figure}[h]
\centering 
\includegraphics[width=3.0in,height=2.0in]{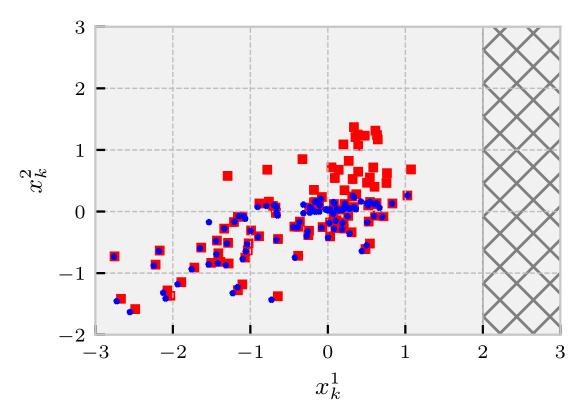} \caption{Values of the Kalman filter conditional mean (red squares) with control $u_k=Kx_0^\star$. The blue dots indicate the values chosen for the candidate state, $x_0^\star,$ used by the controller.
\label{fig:Linear_SS}}
\end{figure}
\begin{figure}[h]
\centering 
\includegraphics[width=3.0in,height=2.0in]{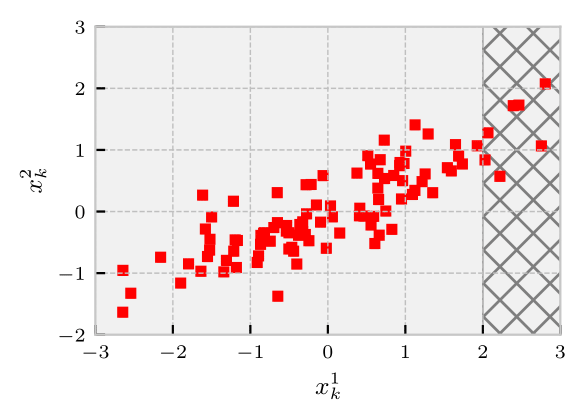} \caption{Values of the Kalman filter conditional mean, $\hat x_{0|0}$, with $u_k=K\hat x_{0|0}$.\label{fig:Linear_CM}}
\end{figure}

Figure~\ref{fig:Linear_xk1} shows: the conditional mean of $x_k^1$, from Figure~\ref{fig:Linear_SS}; its two-sigma intervals propagated by the Kalman filter (the square root of the (1,1) entry of the conditional covariance); and the true state. The shaded area is two standard deviations, or $95\%$ confidence interval. Since the tolerance used in the algorithm for this example is $\epsilon=10\%$, the solution is conservative. 
\begin{figure}[h]
\centering 
\includegraphics[width=3.0in,height=2.0in]{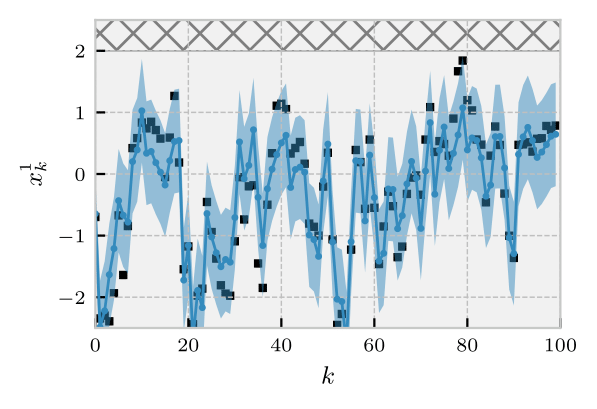} \caption{The dark blue line is the conditional mean of $x_k^1$, the shaded area is the two standard deviations about the conditional mean, and the true state is shown as black squares.\label{fig:Linear_xk1}}
\end{figure}

\section{Conclusion}
The State Selection Algorithm requires a non-empty set of initial states $\mathbb X_0^\epsilon$ to run. The estimator, a particle filter in this case, and the fixed controller $\kappa$, have their roles in the feasibility of the algorithm, as we discussed earlier. A particle filter that is prone to depletion restricts the variability of $\Xi$, and is a poor approximant to the Bayesian filter. While a controller, say $\kappa=0$, eliminates all control sensitivity to the initial state and to the $w'_k$, as is shown in the first example of Section~\ref{Section:examples}. Such a controller would also extirpate the variation seen in the averaged cost.

While replacing the hard probabilistic constraints with soft ones can avoid infeasibility issues, at least in an algorithmic sense. A more concrete understanding of the problem is required, and hence our future work is towards
\begin{itemize}
    \item Enriching the proposed distribution of control sequences with more options, potentially by having $w_k'$ of a different statistics than $w_k''$.
    \item Identifying more practical characteristics required of the controller $\kappa$.
    \item Exploring parameter estimation for control. This can be done by a simple state augmentation, and the provided framework, whether from the particle filter side, or the State Selection Algorithm, can naturally adapt.
\end{itemize}

\bibliographystyle{IEEEtran} %\bibliographystyle{ieeetr}        % Include this if you use bibtex 
\bibliography{SE4MPC}

\appendix
\section{Appendix}
\subsection*{Derivation of $\mathcal{A}_1$ and $\mathcal{A}_2$}
Substituting the states in terms of their means and errors, from (\ref{means}), in equation (\ref{StateCandidacyMeasureLinearTrace})
\begin{align}
    &J_c(x_0')=\E_{x_0''}\E_{W'}\E_{W''} \bigg (\sum_{k=0}^{N-1} \Big[ \trace (Q \hat x_k''(\hat x_k'')^T) \nonumber\\
    &\hskip 7mm+\trace(K^T R K \hat x_k'(\hat x_k')^T)+\trace (Q \tilde x_k''(\tilde x_k'')^T)+\nonumber\\
    &\hskip 7mm+\trace(K^T R K \tilde x_k'(\tilde x_k')^T)\Big] + \trace(Q_N\hat x_N''(\hat x_N'')^T)\nonumber\\
    &\hskip 7mm+\trace(Q_N\tilde x_N''(\tilde x_N'')^T)\bigg ), \label{StateCandidacyMeasureLinearTrace1}
\end{align}
where the cross-terms (means and errors) are ignored due to the states' errors' zero means, eventually by expectation.

The error covariances are not functions of $x_0'$, and thus can be replaced by some constant $C$ in the cost in (\ref{StateCandidacyMeasureLinearTrace1}) without altering the minimizer. Hence, up to an additive constant
\begin{align}
    J_c(x_0')&=\sum_{k=0}^{N-1} \Big[ (\hat x_k'')^TQ \hat x_k'' +(\hat x_k')^T K^T R K \hat x_k'\nonumber \\
    &\hskip -5mm+ (\hat x_N'')^T Q_N\hat x_N'' \bigg ] + C,\label{StateCandidacyMeasureLinearTrace2}
\end{align}
where the traces are returned to their quadratic forms. Substituting $\hat x_k'$ and $\hat x_k''$ from (\ref{means}) in (\ref{StateCandidacyMeasureLinearTrace2}) yields
\begin{align}
    J_c(x_0')&=\sum_{k=0}^{N-1} \Big[ (\hat x_0'')^T(F^k)^T Q F^k\hat x_0'' \nonumber \\
    &\hskip -5mm+2(\hat x_0'')^T(F^k)^T Q \Psi_{k-1}x_0'\nonumber\\
    &\hskip -5mm+(x_0')^T\Psi_{k-1}^TQ \Psi_{k-1}x_0'\nonumber\\
    &\hskip -5mm+( x_0')^T(F_K^k)^T K^T R K F_K^k  x_0'\nonumber \\
    &\hskip -5mm+ (\hat x_0'')^T(F^N)^T Q F^N\hat x_0'' \nonumber \\
    &\hskip -5mm+2(\hat x_0'')^T(F^N)^T Q \Psi_{N-1}x_0'\nonumber\\
    &+(x_0')^T\Psi_{N-1}^TQ \Psi_{N-1}x_0' \bigg ] + C.
\end{align}
All terms which are constants with respect to $x_0'$ can be added to $C$ to be $C_1$,
\begin{align}
    J_c(x_0')&=\sum_{k=0}^{N-1} \Big[ 2(\hat x_0'')^T(F^k)^T Q \Psi_{k-1}x_0'\nonumber\\
    &\hskip -5mm+(x_0')^T\Psi_{k-1}^TQ \Psi_{k-1}x_0'\nonumber\\
    &\hskip -5mm+( x_0')^T(F_K^k)^T K^T R K F_K^k  x_0'\nonumber \\
    &\hskip -5mm+2(\hat x_0'')^T(F^N)^T Q_N \Psi_{N-1}x_0'\nonumber\\
    &+(x_0')^T\Psi_{N-1}^T Q_N \Psi_{N-1}x_0' \bigg ] + C_1,\label{StateCandidacyMeasureLinearTrace3}
\end{align}
or in a compact form, ignoring additive constants
\begin{align} \label{PenaltyLinear}
    J_c(x_0')&=(x_0')^T \mathcal{A}_1 x_0'+(\hat x_0'')^T \mathcal{A}_2 x_0',
\end{align}
where
\begin{align}
    \mathcal{A}_2&=\sum_{k=0}^{N-1} \Big[ 2(F^k)^T Q \Psi_{k-1}\Big] + 2(F^N)^T Q_N \Psi_{N-1}, \label{matrixA1} \\
    \mathcal{A}_1&= \sum_{k=0}^{N-1} \Big[ \Psi_{k-1}^TQ \Psi_{k-1}+(F_K^k)^T K^T R K F_K^k\Big] + \nonumber \\ & \hskip 15mm \Psi_{N-1}^T Q_N \Psi_{N-1}.
\end{align}

\subsection*{Computation of $\kappa_2$}
Let $V$ be the value function, the optimal cost, of the infinite-horizon discounted-cost whose stage costs are $\ell_k(x_k,u_k)=x_k^Tx_k+u_k^2$ and with a discount factor $\gamma=0.9$ \cite{bertsekas2012dynamic}. We define $\kappa_2$ as the corresponding optimal control. That is, the Dynamic Programming Equation is
\begin{align} \label{ValueIteration}
    V(x_k)=x_k^Tx_k+\min_{u_k \in \bar{\mathbb{U}}(x_k)} \left \{u_k^2+ \gamma \E_{w_k} V(f(x_k,u_k,w_k)) \right\},
\end{align}
where $\bar{\mathbb{U}}(x_k)$ is the state dependent input constraint set, defined as $\bar{\mathbb{U}}(x_k)=\{ u \in \mathbb{U} \mid\, \mathbb{P}(f(x_k,u_k,w_k) \in \mathcal{L})<\epsilon\}$. 

The controller $\kappa_2$ can be computed via the Value Iteration over the finite gridded state space and input space Markov Decision Process, which if the grid size of these spaces is large enough, $\kappa_2$ is optimal with respect to the original infinite state-space problem \cite{bertsekas1975convergence}.

Following the dynamic programming approach in \cite{rust1997using}, uniformly randomized grids of $4000$ points over $[-10,10]\times[-5,15]$ in the state space and $50$ points in $[-3,3]$, the control space $\mathbb U$, are generated. The continuous probability in the definition of $\bar{\mathbb{U}}(x)$ is replaced by discrete over the points of the grid which are inside $\mathcal{L}$, and the expectation in \eqref{ValueIteration} by its sample average over the grid. Value Iteration was conducted over the grid until convergence, which is guaranteed. The resulting optimal control
\begin{align}
    \kappa_2(x_k)=\argmin_{u_k\in\bar{\mathbb{U}}(x_k)} \left \{u_k^2+ \gamma \E_{w_k} V(f(x_k,u_k,w_k)) \right\},
\end{align}
can be computed by replacing the expectation with its sample average over the grid. A point $\zeta$ on the grid such that $\kappa(\zeta)$ is infeasible is included into the grid defining $\mathcal{L}$. 

After evaluating $\kappa_2$, it is then fitted into a piecewise linear surface over the region $[-10,10]\times[-15,5]$ of the state space, and $\kappa_2(x)$ is then acquired by $2$D linear interpolation.
\end{document}